\documentclass[11pt,a4paper]{article}

\usepackage[headinclude]{typearea}
\usepackage[english]{babel}           
\usepackage[T1]{fontenc}
\usepackage{scrtime}
\usepackage{amsmath,amsthm,amsfonts,amssymb}
\usepackage {color}                                  
\usepackage {pifont}                                 
\usepackage {multicol,multirow}                               
\usepackage[numbers]{natbib}                                                     
\usepackage[normalem]{ulem}                                                     
\usepackage [scaled=.90]{helvet}                     
\usepackage {courier}                                
\usepackage {graphicx}                               
\usepackage {array}                                  
\usepackage {longtable}                              
\usepackage{fancyvrb}
\usepackage{enumerate} 
\usepackage{ifpdf}
\usepackage{caption}
\usepackage{mathrsfs}

\usepackage{setspace}

\usepackage{marvosym}

\usepackage{mathtools}
\mathtoolsset{showonlyrefs} 

\newtheorem{theorem}{Theorem}[section]

\newtheorem{lemma}[theorem]{Lemma}

\newtheorem{proposition}[theorem]{Proposition}

\theoremstyle{definition}
\newtheorem{assumption}{Assumption}[section]
\newtheorem{definition}[theorem]{Definition}
\newtheorem{example}[theorem]{Example}

\newcommand{\exclude}[1]{}

\newcommand{\proj}{p}

\newcommand{\1}{{\mathbf{1}}}

\newcommand{\N}{{\mathbb{N}}}
\renewcommand{\P}{{\mathbb{P}}}

\newcommand{\R}{{\mathbb{R}}}

\newcommand{\F}{\mathbb{F}}

\definecolor{darkgreen}{rgb}{0,0.5,0}
\definecolor{lightgreen}{rgb}{0.5,0.9,0.5}
\definecolor{magenta}{rgb}{0.75,0,0.25}

\definecolor{violet}{rgb}{0.25,0,0.75}

\newcommand{\hypsurf}{\Theta}

\newcommand{\tr}{\operatorname{tr}}

\renewcommand{\P}{{\mathbb P}}

\newcommand{\cF}{{\cal F}}

\newcommand{\cU}{{\cal U}}

\newcommand{\be}{\begin{equation}}
\newcommand{\ee}{\end{equation}}
\newcommand{\bea}{\begin{eqnarray}}
\newcommand{\eea}{\end{eqnarray}}
\newcommand{\beast}{\begin{eqnarray*}}
\newcommand{\eeast}{\end{eqnarray*}}
\newcommand{\bproof}{\begin{proof}}
\newcommand{\eproof}{\end{proof}}

\title{
Existence and uniqueness of solutions of SDEs with discontinuous drift and finite activity jumps
}

\author{Pawe{\l} Przyby{\l}owicz \and Michaela Sz\"olgyenyi \and Fanhui Xu}
\date{Preprint, January 2021}

\begin{document}

\maketitle


\begin{abstract}
In this letter we prove existence and uniqueness of strong solutions to multi-dimensional SDEs with discontinuous drift and finite activity jumps.\\

\noindent Keywords: multi-dimensional jump-diffusion stochastic differential equations, finite activity jumps, discontinuous drift, existence and uniqueness, strong solutions\\
Mathematics Subject Classification (2020): 60H10
\end{abstract}

\section{Introduction}\label{sec:intro}

On a filtered probability space $(\Omega,\cF,\F=(\F_t)_{t\in[0,T]},\P)$ that satisfies the usual conditions we consider a time-homogeneous stochastic differential equation (SDE) with finite activity jumps,
\begin{align}\label{eq:SDE}
dX_t &= \mu(X_{t-})  dt + \sigma (X_{t-}) dW_t+\int_{\R} \rho(X_{t-},y) \nu(dy,dt), \quad t\in[0,T], \quad X_0=x_0\in\R^d,
 \end{align}
where  $\mu\colon\R^d\to \R^d$, $\sigma\colon\R^d\to \R^{d\times d}$, and $\rho\colon\R^d \times \R \to \R^d$ are measurable functions, $T\in(0,\infty)$, and $W=(W_t)_{t\in[0,T]}$ is a
$d$-dimensional standard Brownian motion. Furthermore, $\nu$ is a Poisson random measure with finite jump intensity, associated with a scalar compound Poisson process. We assume that both $W$ and $\nu$ are adapted to $\F$ and independent of each other.

In case the coefficient $\mu$ is Lipschitz, it is known that SDE \eqref{eq:SDE} admits a unique strong solution, cf.~\cite[p.~79, Theorem 117]{situ2005} or \cite[Theorem 1.9.3]{PBL}.
However, there are applications where jump SDEs with discontinuous drift appear. For example control actions on energy markets often lead to discontinuities in the drift of the controlled process, cf., e.g., \cite{sz2016a,shardin2017}. So to be able to study control problems on energy markets in the most common models for energy prices, i.e.~jump SDE models, existence and uniqueness needs to be settled.

\emph{In this paper we provide sufficient conditions for the existence and uniqueness of strong solutions to \eqref{eq:SDE}, where we allow $\mu$ to be discontinuous.}

In the jump-free case, existence and uniqueness results for strong solutions to SDEs with discontinuous drift have been proven in the classical papers \cite{Zvonkin1974,Veretennikov1981,Veretennikov1982,Veretennikov1984} and in the more recent papers \cite{KR,sz14,sz2016a,sz15,XZ,sz2017a}.
In the case of presence of jumps in the driving process, existence and uniqueness of global strong solutions is proven in \cite[Theorem 3.1]{PS20} for the scalar case with Poisson jumps.
In \cite{Z2013} they prove existence and uniqueness of a maximal local solution for SDEs driven by a standard symmetric $\alpha$-stable process under certain integrability and boundedness conditions. In \cite{XZ1} existence and uniqueness of global strong solutions is obtained under Sobolev regularity assumptions.

In this letter we prove an existence and uniqueness result for multidimensional SDEs with discontinuous drift and general finite activity jumps.
For the proof we apply a tailor-made transformation method. As a benefit, integrability assumptions on the coefficients are not required.

\section{Assumptions}\label{sec:ass}

In the following we denote by $\|\cdot\|$ the respective Euclidean or matrix norm.

\begin{assumption}\label{ass:jump}
The jump measure $\nu$ is a Poisson random measure with finite jump intensity, associated with a compound Poisson process $L=(L_t)_{t\in[0,T]}$, that is $L_t=\sum_{k=1}^{N_t}\xi_k$, where $N=(N_t)_{t\in[0,T]}$ is a Poisson process and $(\xi_k)_{k\in\N}$ is a family of iid random variables independent of $N$ with associated distribution  $\psi$ that has finite second moments and $\mathbb{P}(\xi_k=0)=0$ for all $k\in\N$. 
\end{assumption}
Since the process $L$ has only a finite number of jumps in any interval $(0,t]$, we get for all $t\in[0,T]$,
\begin{equation}
	\int_0^t\int_{\R} \rho(X_{s-},y) \nu(dy,ds)=\sum_{0<s\leq t}\rho(X_{s-},\Delta L_s)\mathbf{1}_D(s)=\sum_{k=1}^{N_t}\rho(X_{\tau_{k}-},\xi_k),
\end{equation}
where $D=\{(\omega,t) \colon \Delta L_t(\omega)=L_{t}(\omega)-L_{t-}(\omega)\neq 0\}$ and $\tau_k$ is the time of the $k$-th jump of $N$, see, for example, \cite[Theorem 2.3.2, p.~22--23]{lukaszdelong}.

For later use, we denote the continuous part of $X$ for all $t\in[0,T]$ by
$$X_t^c=X_t-\int_0^t\int_{\R} \rho(X_{s-},y) \nu(dy,ds).$$

In order to define assumptions on the drift coefficient, we recall the following definitions from differential geometry.

\begin{definition}[\text{\cite[Definitions 3.1 and 3.2]{sz2017a}}]
Let $A\subseteq \R^d$.
\begin{enumerate}
\item For a continuous curve $\gamma\colon [0,1]\to \R^d$, let  $\ell(\gamma)$ denote its length, i.e.
\[
\ell(\gamma)=\sup_{n \in \mathbb{N}, \, 0\le t_1<\ldots<t_n\le 1}\sum_{k=1}^n \|\gamma(t_k)-\gamma(t_{k-1})\|.
\] 
The {\em intrinsic metric} $\varrho$ on $A$ 
is given by 
\[
\varrho(x,y):=\inf\{\ell(\gamma)\colon \gamma\colon[0,1]\to A \text{ is a continuous curve satisfying } \gamma(0)=x,\, \gamma(1)=y\},
\]
where $\varrho(x,y):=\infty$, if there is no continuous curve from $x$
to $y$.  
\item Let $f\colon A\to \R^m$ be a function.
We say that $f$ is {\em intrinsic Lipschitz}, if it is Lipschitz w.r.t. the
intrinsic metric on $A$, i.e.~if there exists a constant $L\in(0,\infty)$ such that for all $x,y\in A$,
\[
 \|f(x)-f(y)\|\le L \varrho(x,y).
\]
\end{enumerate}
\end{definition}

\begin{definition}[\text{\cite[Definition 3.4]{sz2017a}}]\label{def:pw-lip}
A function $f\colon \R^d\to\R^m$ is {\em piecewise Lipschitz}, if
there exists a hypersurface $\hypsurf\subseteq \R^d$ with finitely many connected 
components and with the property, that
the restriction $f|_{\R^d\backslash \hypsurf}$ is intrinsic Lipschitz.
We call $\hypsurf$ an {\em exceptional set} of $f$,
and we call
\[
\sup_{x,y\in \R^d\backslash\hypsurf}\frac{\|f(x)-f(y)\|}{\varrho(x,y)}
\]
the \emph{piecewise Lipschitz constant} of $f$.
\end{definition}

We will assume that the drift $\mu$ is piecewise Lipschitz with exceptional set $\hypsurf$, where $\hypsurf$ is a fixed, sufficiently regular hypersurface, see Assumption \ref{ass:ex-un} below.

If $\hypsurf\in C^4$, locally there exists a unit normal vector, that is a continuously differentiable $C^3$
function $n\colon D\subseteq\Theta\to \R^d$ (cf.~\cite[Theorem 4.1]{dudekholly}) such that for every $\zeta\in D$, $\|n(\zeta)\|=1$, and $n(\zeta)$ is orthogonal to the tangent space of $\hypsurf$ in $\zeta$. Recall the following definition.

\begin{definition}\label{def:positivereach}
Let $\hypsurf \subseteq \R^d$. 
\begin{enumerate}
\item \label{def:ucpp} 
For $\varepsilon\in(0,\infty)$ an environment $\hypsurf^\varepsilon$ of $\hypsurf$ is said to have the 
{\em unique closest point property}, if for every $x\in \R^d$
with $d(x,\hypsurf)<\varepsilon$ there is a unique $\proj\in \hypsurf$ with
$d(x,\hypsurf)=\|x-\proj\|$.
Therefore, we can define a mapping 
$\proj\colon\hypsurf^{\varepsilon}\to \hypsurf$ assigning to each $x$
the point $\proj(x)$ on $\hypsurf$, which is closest to $x$.
\item 
A set $\hypsurf$ is said to be of {\em positive reach}, if there exists  
$\varepsilon\in(0,\infty)$ such that  $\hypsurf^\varepsilon$ has the  
unique closest point property.
The {\em reach} $r_{\hypsurf}$ of $\hypsurf$ is the supremum 
over all such $\varepsilon$ if such an $\varepsilon$ exists, and 
0 otherwise. 
\end{enumerate}

\end{definition}

\begin{assumption}\label{ass:ex-un}We make the following assumptions on the functions $\mu,\sigma,\rho$:
\begin{description}
\item[(ass.~$\mu$)] The drift coefficient $\mu\colon\R^d \to \R^d$ is piecewise Lipschitz;
its exceptional set $\hypsurf$ is a $C^4$-hypersurface with reach $r_{\hypsurf}> \varepsilon_0$ for some $\varepsilon_0\in(0,\infty)$ and every unit normal vector $n$ of $\hypsurf$ has bounded second and third derivatives.
\item[(ass.~$\sigma$)] The diffusion coefficient $\sigma\colon\R^d\to\R^{d\times d}$ is Lipschitz.\item[(loc.~bound)] The coefficients $\mu$ and $\sigma$ satisfy
$$  \sup_{x \in \hypsurf^{\varepsilon_0}} (\|\mu(x)\| + \|\sigma(x)\|) < \infty;     $$
\item[(non-parallelity)]  There exists a constant $c_0\in(0,\infty)$ such that $\|\sigma(\xi)^\top n(\xi)\|\ge c_0$ for
all $\xi\in \hypsurf$;
\item[(techn.~ass.)] The function
\begin{align}\label{eq:alphad}
\alpha\colon \hypsurf \rightarrow \R^d, \qquad \alpha(x)=\lim_{h\to 0+}\frac{\mu(x-h n(x))-\mu(x+hn(x))}{2 \|\sigma(x)^\top n(x)\|^2}
\end{align} is well defined, bounded, and
belongs to $C^3_b(\hypsurf;\mathbb{R}^d)$.
\item[(ass.~$\rho$)] There exists $c_{\rho}\in (0,\infty)$ such that the jump coefficient $\rho\colon\R^d \times \R \to\R^d$ satisfies for all $x\in\mathbb{R}^d$,
\begin{equation*}
	 \int_\R \|\rho(x,y)\|^2 \psi(dy) \le c_{\rho}(1+\|x\|^2)
\end{equation*}
and for all $x,z\in\R^d$,
\begin{equation*}
	\int_\R \|\rho(x,y)-\rho(z,y)\|^2 \psi(dy) \le c_{\rho}\|x-z\|^2,
\end{equation*}
where as before $\psi$ is the distribution of $\xi_1$.
\end{description}
\end{assumption}

\section{Preliminary results}

As a preliminary for the proof of our main result on existence and uniqueness of solutions of SDE \eqref{eq:SDE}, we need to adopt the transformation introduced in \cite{sz2017a} and recall its properties. Furthermore, we require an It\^o-type formula for this transform and its inverse.

\subsection*{The transform}
In the proof of our main result, we will apply a transform $G\colon \R^d\to \R^d$ that has the property that
the process formally defined by $Z=G(X)$ satisfies an SDE with 
Lipschitz coefficients and therefore has a solution by classical results, see \cite[Theorem 1.9.3]{PBL}.

The function $G$ is chosen so that it impacts the coefficients of the SDE \eqref{eq:SDE} only locally around the points of discontinuity of the drift.
This behaviour is ensured by incorporating a bump function into $G$. Adopted from \cite{sz2017a}, $G$ is given by
\begin{align}\label{eq:G}
G(x)=\begin{cases}
 x+\varphi(x) \alpha(\proj(x)),&  x\in \hypsurf^{\varepsilon_0},\\
x, & x\in \R^d\backslash \hypsurf^{\varepsilon_0},
\end{cases}
\end{align}
with $r_{\hypsurf}>\varepsilon_0>0$, see 
Assumption \ref{ass:ex-un},
$\alpha$ as in Assumption \ref{ass:ex-un}, and 
\begin{align}  \label{eq:varphi}
 \varphi(x)=n(\proj(x))^\top(x-\proj(x))
\|x-\proj(x)\|\phi\left(\frac{\|x-\proj(x)\|}{c}\right),                                                                                                                                       
 \end{align}
with a constant $c\in(0,\infty)$ and $\phi\colon \R \to \R$,
\begin{align*}
\phi(u)=
\begin{cases}
(1+u)^4 (1-u)^4, & |u|\le 1,\\
0, & |u|> 1.
\end{cases}
\end{align*}

We recall the following properties of $G$. Thereby, we denote $G':=\nabla G$ and by $G''$ the Hessian of $G$.
\begin{lemma}[\text{\cite[Lemma 2.6]{NSS19}}]
\label{lemG}
Let Assumption \ref{ass:ex-un} hold. Then we have
\begin{enumerate}[(i)]\setlength{\itemsep}{0em}
 \item\label{it:C1} $G\in C^1(\mathbb{R}^d,\mathbb{R}^d)$;
 \item\label{it:GpwLip} $G'$ is Lipschitz, $G''$ exists on $\mathbb{R}^d \setminus \hypsurf$ and is piecewise Lipschitz with exceptional set $\hypsurf$;
 \item\label{it:boundedder} $G'$ and $G''$ are bounded;
 \item\label{it:Ginv} for $c$ sufficiently small, see \cite[Lemma 1]{LS19}, $G$ is globally invertible;
 \item\label{it:GLip} $G$ and $G^{-1}$ are Lipschitz continuous.
\end{enumerate}
\end{lemma}

\subsection*{A change of variable formula}

We require a multidimensional It\^o-type formula, which is in a very specific sense slightly more general than \cite[Theorem 3.1]{peskir2007}, and which is, up to our best knowledge, novel. Note that if the It\^o-type formula holds for the scalar components of $G=(G_1,\dots,G_d)$ and $G^{-1}=(G^{-1}_1,\dots,G^{-1}_d)$, then it also holds for $G,G^{-1}$.

\begin{proposition}\label{propito}
Let $G$ be as in \eqref{eq:G} and let Assumption \ref{ass:ex-un} be satisfied.
Let $Y=(Y^1,\dots,Y^d)$ be a semimartingale with $Y_0=y_0$ and with finite variation jumps that satisfies for all $s\in(0,T]$,
$\P(Y_s\in\hypsurf)=0$.
Then for all $f\in\{G_1,\dots,G_d,G^{-1}_1,\dots,G^{-1}_d\}$ and all (stopping times) $t\in[0,T]$ it holds that
\begin{equation}\label{ito}
\begin{aligned}
 f(Y_t)=f(y_0)&+\sum_{k=1}^d \int_0^t \frac{\partial f}{\partial y_k} (Y_{s-}) d(Y^c_s)^k + \frac{1}{2} \sum_{k,j=1}^d \int_0^t \frac{\partial^2 f}{\partial y_k\partial y_j} (Y_{s-}) d[(Y^c_s)^k,(Y^c_s)^j] 
 \\&+
  \sum_{0<s\le t} \big( f(Y_s)-f(Y_{s-}) \big).
 \end{aligned}
\end{equation}
\end{proposition}

\begin{proof}
On $\R^d\backslash \hypsurf$, $f \in C^2$, hence \eqref{ito} is settled by \cite[eq.~(3.4), p.~81]{peskir2007}. Therefore, It\^o's formula holds for $y_0\in\R^d\backslash \hypsurf$ until the first time $Y$ hits $\hypsurf$. So we may restrict our analysis to the case $y_0\in \hypsurf$.


There exists an open set $\cU \subseteq \hypsurf^{\varepsilon_0}$ with $y_0\in\cU$ such that there exists $\psi^{\cU}\in C^2( \R^{d-1}, \R)$ with the property that all $y=(y_1,\dots,y_{d-1},y_d)\in\hypsurf \cap \cU$ can be represented as
$y=(y_1,\dots,y_{d-1},\psi^\cU(y_1,\dots,y_{d-1}))$,
 or such that we can rotate the state space in a smooth way so that the representation holds. Hence without loss of generality we assume that such a rotation is not necessary. This means that the set of points where $f'$ is not differentiable is locally given by some $\{y^d=\psi^\cU(y_1,\dots,y_{d-1})\} $.

Since $\psi^{\cU}\in C^2( \R^{d-1}, \R)$, it holds that $\psi^\cU(Y^1,\dots,Y^{d-1})$ is a semimartingale.
Due to the properties of $f$ from Lemma \ref{lemG}, on these $\cU$ the It\^o-type formula \cite[Theorem 3.1]{peskir2007} holds. Since $f'$ is a continuous function, the first order terms in \cite[equation (3.6)]{peskir2007} reduce to the form of the first order terms in \eqref{ito}.
The same holds for all mixed derivatives except $ \frac{\partial^2 f}{\partial y_d^2} $.

Furthermore, since $f''$ is bounded, we may apply the local time formula from \cite[Chapter IV, Corollary 1, p.~219]{protter2005}, which says that formally the integral of $f''\cdot \1_{\{(Y_s)^d=\psi^\cU(Y_s^1,\dots,Y_s^{d-1})\}}$ with respect to the quadratic variation of $Y^d$ from \cite[equation (3.6)]{peskir2007} can be expressed as space integral of the same function times a local time.
Since for all $s\in(0,t]$, $\P(Y_s\in\hypsurf)=0$, we have
\begin{align*}
\P((Y_s)^d=\psi^\cU(Y_s^1,\dots,Y_s^{d-1}))=0.
\end{align*}
So together with the local time formula from \cite[Chapter IV, Corollary 1, p.~219]{protter2005} we have that the second order terms in \cite[equation (3.6)]{peskir2007} reduce to the form of the second order terms in \eqref{ito}.
\end{proof}

\section{Existence and uniqueness result}

In this section we are going to prove our main result.

\begin{theorem}
\label{th:exun}
Let Assumptions \ref{ass:jump} and \ref{ass:ex-un} hold.
Then the SDE \eqref{eq:SDE} has a unique global strong solution.
\end{theorem}

\begin{proof}
The idea of the proof is to construct a process $Z$ for which we can verify that $G^{-1}(Z)$ solves \eqref{eq:SDE}.
For this we will first heuristically apply the It\^o-type formula from Proposition \ref{propito} to $G$ componentwise. Observe that
\begin{align*}
	\sum_{0<s\leq t} G'(X_{s-})\Delta X(s)&=\sum_{0<s\leq t}G'(X_{s-})\rho(X_{s-},\Delta L_s)\mathbf{1}_D(s)\notag\\
	&=\int\limits_0^t\int\limits_{\mathbb{R}}G'(X_{s-})\rho(X_{s-},y)\nu (dy,ds)
\end{align*}
and
\begin{align*}
	\int\limits_0^t G'(X_{s-})dX_s&=\int\limits_0^t G'(X_{s-})\mu(X_{s-})ds
	+\int_0^t  G'(X_{s-}) \sigma(X_{s-}) dW_s\notag\\
	&\quad+\int\limits_0^t\int\limits_{\mathbb{R}}G'(X_{s-})\rho(X_{s-},y)\nu (dy,ds).
\end{align*}
Combining this with Proposition \ref{propito} yields
\begin{align*}
	G(X_t)&=G(x_0) + \int_0^t \left(G'(X_{s-})\mu(X_{s-}) +\frac{1}{2}\tr \left[\sigma(X_{s-})^\top G''(X_{s-})\sigma(X_{s-})\right]\right) ds
	\\&\quad
+ \int_0^t  G'(X_{s-})\sigma(X_{s-}) dW_s+\sum_{0<s\le t} \left(G(X_s)-G(X_{s-})\right).
\end{align*}
Observe that
\begin{align*}
	\sum_{0<s\le t} \left(G(X_s)-G(X_{s-})\right) = \int_0^t\int_\R \left(G(X_{s-}+\rho(X_{s-},y))-G(X_{s-})\right) \nu(dy,ds).
\end{align*}
So we have
\begin{align*}
	G(X_t)&=G(x_0) + \int_0^t \left(G'(X_{s-})\mu(X_{s-}) +\frac{1}{2}\tr \left[\sigma(X_{s-})^\top G''(X_{s-})\sigma(X_{s-})\right]\right) ds
	\\&\quad
	+ \int_0^t  G'(X_{s-})\sigma(X_{s-}) dW_s 
	+ \int_0^t\int_\R \left(G(X_{s-}+\rho(X_{s-},y))-G(X_{s-})\right) \nu(dy,ds).
\end{align*}
Now we define the process $Z=G(X)$, which for all $t\in[0,T]$ is given by
\begin{equation}
\label{eq:Z}
	Z_t=G(x_0) + \int_0^t \tilde\mu (Z_{s-})ds+\int_0^t\tilde\sigma(Z_{s-})dW_s
	+
	\int_0^t\int_\R \tilde \rho(Z_{s-},y) \nu(dy,ds),
\end{equation}
where for all $y\in\R$, $z\in\mathbb{R}^d$ 
\begin{align}
\tilde\mu(z)&=G'(G^{-1}(z))\mu(G^{-1}(z)) +\frac{1}{2}\tr \left[\sigma(G^{-1}(z))^\top G''(G^{-1}(z))\sigma(G^{-1}(z))\right],\\
\tilde\sigma(z)&= G'(G^{-1}(z))\sigma(G^{-1}(z)), \\
\label{t_rho_def}
\tilde \rho(z,y)&=G\Bigl(G^{-1}(z)+\rho(G^{-1}(z),y)\Bigr)-G(G^{-1}(z))=G\Bigl(G^{-1}(z)+\rho(G^{-1}(z),y)\Bigr)-z.
\end{align}

In \cite{sz2017a} it is shown that $\tilde\mu$ and $\tilde\sigma$ are Lipschitz. Due to the global Lipschitz continuity of $G$ and $G^{-1}$, the linear growth of $G$, $G^{-1}$, and $\rho$, and the finiteness of the second moment of $\xi_1$ we get that there exists $c_{\tilde{\rho}}\in(0,\infty)$ such that for all $z,z_1,z_2\in\mathbb{R}$,
\begin{align*}
	&\int\limits_{\mathbb{R}}\|\tilde\rho(z_1,y)-\tilde\rho(z_2,y)\|^2\psi(dy)\leq c_{\tilde{\rho}}\|z_1-z_2\|^2,\\
	&\int\limits_{\mathbb{R}}\|\tilde\rho(z,y)\|^2\psi(dy)\leq c_{\tilde{\rho}}(1+\|z\|^2).
\end{align*} 
Hence, the SDE for $Z$, that is \eqref{eq:Z} with initial condition $Z(0)=G(\xi)$, has a unique global strong solution by, for example, \cite[Theorem 1.9.3]{PBL}.

It remains to verify that the process $X$ defined by $X=G^{-1}(Z)$ solves \eqref{eq:SDE}.
For this, first observe that $(G^{-1})'(z)=1/G'(G^{-1}(z))$ is absolutely continuous since it is Lipschitz.
Moreover, $G^{-1}(Z_{t-})=\lim\limits_{s\to t-}G^{-1}(Z_s)=X_{t-}$. Hence, \eqref{t_rho_def} gives
\begin{equation}
	Z_{t-}+\tilde\rho(Z_{t-},y)=G(G^{-1}(Z_{t-})+\rho(X_{t-},y)).
\end{equation}
This implies that
\begin{equation}
\label{gg_1}
	G^{-1}(Z_{t-}+\tilde \rho (Z_{t-},y))-G^{-1}(Z_{t-})=\rho(X_{t-},y).
\end{equation}
Furthermore, the non-parallelity condition of Assumption \ref{ass:ex-un} carries over to $\tilde \sigma$, ensuring that for all $s\in(0,t]$,
$\P(Z_s\in\hypsurf)=0$
(for the argument cf.~also \cite[equation (4.8)]{peskir2007}).
Therefore, Proposition \ref{propito} applied to $G^{-1}$ yields
\begin{align*}
	G^{-1}(Z_t)&=G^{-1}(Z_0) + \int_0^t \left((G^{-1})'(Z_{s-}) \tilde\mu(Z_{s-})+\frac{1}{2}\tr \left[\tilde\sigma(Z_{s-})^\top G''(Z_{s-})\tilde\sigma(Z_{s-})\right]\right) ds
	\\&\quad
+ \int_0^t (G^{-1})'(Z_{s-}) \tilde\sigma(Z_{s-})  dW_s+\sum_{0<s\le t} \left(G^{-1}(Z_s)-G^{-1}(Z_{s-})\right),
\end{align*}
where by \eqref{gg_1},
\begin{align*}
	&\sum_{0<s\le t} \left(G^{-1}(Z_s)-G^{-1}(Z_{s-})\right) \\
	&=\sum_{0<s\leq t}\left(G^{-1}(Z_{s-}+\tilde\rho(Z_{s-},\Delta L_s)\mathbf{1}_D(s))-G^{-1}(Z_{s-})\right)\\
	&=\sum_{0<s\leq t}\left(G^{-1}(Z_{s-}+\tilde\rho(Z_{s-},\Delta L_s))-G^{-1}(Z_{s-})\right)\mathbf{1}_D(s)\\
	&= \int_0^t\int_\R \left(G^{-1}(Z_{s-}+\tilde\rho(Z_{s-},y))-G^{-1}(Z_{s-})\right) \nu(dy,ds)
	\\&= \int_0^t\int_\R \rho(X_{s-},y) \nu(dy,ds).
	\end{align*}
So for $X=G^{-1}(Z)$ we have $X_0=G^{-1}(Z_0)=G^{-1}(G(x_0))=x_0$ and for all $t\in(0,T]$,
\begin{align*}
dX_t&=\Bigl((G^{-1})'(Z_{t-})\tilde\mu (Z_{t-})+\frac{1}{2}\tr \left[\tilde\sigma(Z_{t-})^\top G''(Z_{t-})\tilde\sigma(Z_{t-})\right]\Bigr)dt+(G^{-1})'(Z_{t-})\tilde\sigma(Z_{t-})dW_t
	\\&\quad
	+\int_{\R}\rho(G^{-1}(Z_{t-}),y)\nu (dy,dt)
	\\&=\mu(X_{t-})dt+\sigma(X_{t-})dW_t+\int_{\R}\rho (X_{t-},y)\nu (dy,dt).
\end{align*}
This closes the proof.
\end{proof}

\section{Examples}

We mention the two most important classes of jump processes that are covered in our setup.

\begin{example}[Poisson process] Let $\xi_k=1$ for all $k\in\N$ and for all $y\in\R$ let $\rho(x,y)=\rho(x)$, i.e.~$\rho$ depends only on the first variable. Therefore,
Theorem \ref{th:exun} holds for the case that the jump process governing \eqref{eq:SDE} is a Poisson process, that is for all $t\in[0,T]$,
\begin{equation}
	\int_0^t\int_{\R} \rho(X_{s-},y) \nu(dy,ds)=\sum_{k=1}^{N_t}\rho(X_{\tau_{k}-})=\int\limits_0^t \rho(X_{s-})dN_s.
\end{equation}

For $d=1$ Theorem \ref{th:exun} is equivalent to \cite[Theorem 3.1]{PS20}.
Hence, in the case of Poisson jumps, Theorem \ref{th:exun} extends \cite[Theorem 3.1]{PS20} to a multi-dimensional setting. 
\end{example}
\begin{example}[Compound Poisson process] Now let $\rho(x,y)=y\cdot\rho(x)$.
Therefore, Theorem \ref{th:exun} also holds for the case that the jump process governing \eqref{eq:SDE} is a compound Poisson process, that is for all $t\in[0,T]$,
\begin{equation}
	\int_0^t\int_{\R} \rho(X_{s-},y) \nu(dy,ds)=\sum_{k=1}^{N_t}\xi_k\cdot\rho(X_{\tau_{k}-})=\sum_{0<s\leq t}\rho(X_{s-})\Delta L_s=\int\limits_0^t \rho(X_{s-})dL_s.
\end{equation}
\end{example}


\section*{Acknowledgements}

M.~Sz\"olgyenyi thanks Gunther Leobacher for useful discussions that improved the paper and Goran Peskir for nice discussions about change of variable formulas at the Stochastic Models and Control conference in Trier, Germany in 2017 and during her research stay in Manchester, UK in the same year.\\

M.~Sz\"olgyenyi is supported by the Austrian Science Fund (FWF) project DOC 78.

P.~Przyby{\l}owicz is supported  by  the  National  Science  Centre, Poland, under project\\ 2017/25/B/ST1/00945.



\vspace{2em}
\centerline{\underline{\hspace*{16cm}}}

 \noindent Pawe{\l} Przyby{\l}owicz \\
Faculty of Applied Mathematics, AGH University of Science and Technology, Al.~Mickiewicza 30, 30-059 Krakow, Poland\\
pprzybyl@agh.edu.pl\\

\noindent Michaela Sz\"olgyenyi \Letter \\
Department of Statistics, University of Klagenfurt, Universit\"atsstra\ss{}e 65-67, 9020 Klagenfurt, Austria\\
michaela.szoelgyenyi@aau.at\\

\noindent Fanhui Xu \\
Department of Mathematical Sciences, Carnegie Mellon University, 5000 Forbes Ave, Pittsburgh, PA 15213, United States\\
fanhuix@andrew.cmu.edu


\end{document}